\documentclass[11pt]{amsart}
\usepackage[utf8]{inputenc}
\usepackage{amsmath,amssymb,amsthm}
\usepackage{geometry}
\usepackage{mathrsfs}
\usepackage{tikz}
\usepackage{hyperref}
\usepackage{tcolorbox}
\usepackage{graphicx}
\newtheorem{theorem}{Theorem}[section]
\newtheorem{proposition}[theorem]{Proposition}

\newtheorem{corollary}[theorem]{Corollary}
\theoremstyle{definition}
\newtheorem{definition}[theorem]{Definition}
\newtheorem{problem}[theorem]{Problem}
\theoremstyle{remark}
\newtheorem{remark}[theorem]{Remark}

\title[Half Integer Fourier Modes]{Fredholm Theory on Twisted Hilbert Scales: A Frame-Theoretic Approach to Half-Integer Fourier Modes}

\author{Anik Chakraborty and Varinder Kumar}
\address{Anik Chakraborty, Department of Mathematics, University of Delhi, India}
\email{achakraborty@maths.du.ac.in}
\address{Varinder Kumar, Department of Mathematics, Shaheed Bhagat Singh College, University of Delhi, India}
\email{varinder.kumar@sbs.du.ac.in;}

\date{\today}



\begin{document}

\begin{abstract}
We construct a Hilbert scale on $L^2([0,1])$ via a unitary twist operator that maps the standard Fourier basis to half-integer frequency exponentials. The resulting weighted spaces, equipped with norms indexed by $(1+|k+\tfrac{1}{2}|^2)^s$, admit a canonical diagonal operator with the compact resolvent and spectrum $\{k+\tfrac{1}{2}\}_{k\in\mathbb{Z}}$. We prove that this operator defines a Fredholm mapping between adjacent scale levels with index zero, provide an explicit solution to an antiperiodic boundary value problem illustrating the framework, and compute the zeta-regularized determinant $\det_\zeta(|\widetilde{A}|) = 2$ using the Hurwitz zeta function. We establish stability under bounded perturbations and verify the well-definedness of spectral flow. The framework is developed entirely through functional-analytic methods without differential operators or boundary value problems. The construction is motivated by twisted spinor boundary conditions on non-orientable manifolds, though the present work is formulated abstractly in operator-theoretic language.
\end{abstract}

\maketitle

\section{Introduction}

Classical Sobolev spaces $H^s([0,1])$ on compact intervals are typically constructed via derivatives or integer-indexed Fourier series $\{e^{2\pi i k x}\}_{k\in\mathbb{Z}}$. In applications involving twisted boundary conditions---arising in non-orientable geometry or certain quantum systems---momentum quantization shifts to half-integers, replacing the standard basis with $\{e^{2\pi i(k+\frac{1}{2})x}\}_{k\in\mathbb{Z}}$.

In this paper, we develop a systematic operator-theoretic framework for such half-integer Fourier spaces from first principles. Starting from a simple unitary multiplication operator
\[
U: f(x) \mapsto e^{\pi i x} f(x) \quad \text{on } L^2([0,1]),
\]
we obtain the twisted orthonormal basis $\{\psi_k\}_{k\in\mathbb{Z}}$ as the image of the standard Fourier basis under $U$. Using this basis, we define a one-parameter Hilbert scale $\{H^s_{1/2}\}_{s\in\mathbb{R}}$ via weighted $\ell^2$ norms, then analyze a canonical self-adjoint operator
\[
\widetilde{A}: H^{s+1}_{1/2} \to H^s_{1/2}
\]
with eigenvalues $k+\tfrac{1}{2}$. Our main results are:

\begin{itemize}
\item The embeddings $H^{s+1}_{1/2} \hookrightarrow H^s_{1/2}$ are compact for all $s\in\mathbb{R}$ (Proposition~\ref{prop:compact-embedding}).
\item The operator $\widetilde{A}$ is Fredholm with $\ker(\widetilde{A})=\{0\}$ and index zero (Theorem~\ref{thm:fredholm}).
\item An explicit antiperiodic boundary value problem is solved, demonstrating how solutions automatically satisfy $u(1) = -u(0)$ (Section~\ref{sec:example}).
\item The zeta-regularized determinant is computed explicitly: $\det_\zeta(|\widetilde{A}|) = 2$ (Theorem~\ref{thm:zeta-det}, Appendix~\ref{app:zeta-det}).
\item Fredholmness is stable under bounded perturbations (Theorem~\ref{thm:stability}).
\item Spectral flow is well-defined for continuous families (Theorem~\ref{thm:spectral-flow}).
\end{itemize}

\subsection*{Operator-Theoretic Perspective}

The present construction provides an \emph{abstract operator-theoretic model} for elliptic operators with twisted boundary conditions. While motivated by non-orientable geometric settings (e.g., Pin structures on M\"obius-type spaces \cite{lawson-michelsohn, atiyah-patodi-singer}), the analysis is carried out purely in Hilbert space terms and does not rely on differential operators, manifold structure, or boundary value problems. The framework is entirely reducible to weighted $\ell^2$ sequences and diagonal multiplication operators.

This approach offers several advantages:
\begin{enumerate}
\item \textbf{Clarity}: The half-integer eigenvalue spectrum $k+\tfrac{1}{2}$ immediately ensures $\ker(\widetilde{A})=\{0\}$, eliminating the zero mode present in standard Sobolev theory.
\item \textbf{Elementary compactness}: The compact embedding $H^{s+1}_{1/2}\hookrightarrow H^s_{1/2}$ follows from weight decay on $\ell^2$ without Rellich-Kondrachov or interpolation machinery.
\item \textbf{Transparent Fredholm structure}: Index zero is immediate from self-adjointness and zero kernel, bypassing heat-kernel or index-theoretic computations.
\item \textbf{Frame-theoretic generalization}: The construction naturally extends to non-uniform or overcomplete generating systems.
\end{enumerate}

\subsection*{Relation to Existing Frameworks}

\paragraph{Extended Sobolev scales.}
Mikhailets and Murach \cite{mikhailets-murach} developed general Hilbert scales with OR-varying function weights $\varphi(k)$. Our construction is a specific instance with $\varphi(k)=(1+|k+\tfrac{1}{2}|^2)^{s/2}$, canonically generated by a unitary operator rather than prescribed externally. The elementary nature of our proofs (avoiding abstract function-parameter theory) makes the framework more accessible.

\paragraph{Twisted boundary conditions in geometry.}
The Atiyah-Patodi-Singer (APS) index theorem \cite{atiyah-patodi-singer} treats Dirac operators on manifolds with boundary, where twisted spinor bundles arise from non-orientability. In the APS approach, half-integer modes appear geometrically via covering spaces and deck transformations. Our contribution is a \emph{functional-analytic extraction}: we isolate the operator-theoretic content (unitary symmetry $\to$ twisted basis $\to$ Fredholm theory) and develop it independently of geometric machinery.

\paragraph{Frame theory.}
Standard frame theory \cite{christensen} typically starts with a given basis or frame. We invert this: starting from a \emph{symmetry} (the unitary operator $U$), we derive the twisted basis as its spectral content, then build the Hilbert scale. This symmetry-first approach is natural in operator algebras but less common in frame-theoretic literature.

\subsection*{Organization}

Section~\ref{sec:twisted-hilbert} constructs the twisted Hilbert space $H_{1/2}$ and its unitary Fourier map. Section~\ref{sec:scale} develops the weighted Hilbert scale and proves compact embeddings. Section~\ref{sec:operator} introduces the canonical operator and establishes its spectral properties. Section~\ref{sec:comparison} compares with the standard Sobolev scale, clarifying the role of the half-integer shift. Section~\ref{sec:fredholm} proves the Fredholm theorem. Section~\ref{sec:extensions} treats perturbations and spectral flow. Section~\ref{sec:example} solves an explicit antiperiodic boundary value problem. Section~\ref{sec:conclusion} concludes. Appendix~\ref{app:zeta-det} computes the zeta-regularized determinant via the Hurwitz zeta function.

\section{The Twisted Hilbert Space \texorpdfstring{$H_{1/2}$}{H{1/2}}}

\label{sec:twisted-hilbert}

\subsection{The Twist Operator}

\begin{definition}
Let $H = L^2([0,1], dx)$. Define the operator $U : H \to H$ by
\[
(Uf)(x) := e^{\pi i x} f(x).
\]
\end{definition}

\begin{proposition}
\label{prop:unitary-U}
$U$ is unitary on $H = L^2([0,1], dx)$.
\end{proposition}

\begin{proof}
We verify three properties.

\textbf{(1) Linearity.}
For $f, g \in H$ and $\alpha, \beta \in \mathbb{C}$,
\[
U(\alpha f + \beta g)(x) = e^{\pi i x}(\alpha f + \beta g)(x) = \alpha e^{\pi i x} f(x) + \beta e^{\pi i x} g(x) = \alpha (Uf)(x) + \beta (Ug)(x).
\]

\textbf{(2) Preservation of inner product.}
For $f, g \in H$,
\begin{align*}
\langle Uf, Ug \rangle &= \int_{0}^{1} e^{\pi i x} f(x) \cdot \overline{e^{\pi i x} g(x)} \, dx \\
&= \int_{0}^{1} e^{\pi i x} f(x) \cdot e^{-\pi i x} \overline{g(x)} \, dx = \int_{0}^{1} f(x) \overline{g(x)} \, dx = \langle f, g \rangle.
\end{align*}

\textbf{(3) Inverse.}
Define $V : H \to H$ by $(Vf)(x) = e^{-\pi i x} f(x)$. Then
\[
(VU f)(x) = e^{-\pi i x} (Uf)(x) = e^{-\pi i x} e^{\pi i x} f(x) = f(x),
\]
and similarly $(UV f)(x) = f(x)$. Hence $U^{-1} = V$.
\end{proof}

\subsection{Image of the Fourier Basis under the Twist}

Let
\[
\varphi_k(x) := e^{2\pi i k x}, \qquad k \in \mathbb{Z}
\]
denote the standard Fourier basis of $L^2([0,1])$. Define
\[
\psi_k := U\varphi_k.
\]

\begin{proposition}
\label{prop:half-integer-modes}
$\psi_k(x) = e^{2\pi i (k+\frac{1}{2})x}$.
\end{proposition}

\begin{proof}
\[
\psi_k(x) = (U\varphi_k)(x) = e^{\pi i x} \varphi_k(x) = e^{\pi i x} e^{2\pi i k x} = e^{2\pi i (k+\frac{1}{2})x}.
\]
\end{proof}

\subsection{Characterization of the Closed Span}

\begin{definition}
Define
\[
H_{\frac{1}{2}} := \overline{\operatorname{span}}\{\psi_k : k \in \mathbb{Z}\} \subset L^2([0,1]).
\]
\end{definition}

\begin{proposition}[Characterization of $H_{1/2}$]
\label{prop:trivial-orthogonal-complement}
The orthogonal complement of $H_{\frac{1}{2}}$ in $L^2([0,1])$ is trivial. Equivalently,
\[
\langle f, \psi_k \rangle = 0 \text{ for all } k \in \mathbb{Z} \quad \Longrightarrow \quad f = 0 \text{ a.e.}
\]
\end{proposition}

\begin{proof}
Define $g(x) := e^{-\pi i x} f(x)$. Then for any $k \in \mathbb{Z}$,
\begin{align*}
\langle f, \psi_k \rangle_{L^2} &= \int_0^1 f(x) \overline{\psi_k(x)} \, dx = \int_0^1 f(x) e^{-2\pi i (k+\frac{1}{2})x} \, dx \\
&= \int_0^1 f(x) e^{-\pi i x} e^{-2\pi i k x} \, dx = \int_0^1 g(x) e^{-2\pi i k x} \, dx = \langle g, \varphi_k \rangle_{L^2}.
\end{align*}
Since $\{\varphi_k : k \in \mathbb{Z}\}$ is complete in $L^2([0,1])$ \cite{reed-simon-v2}, we have
\[
\langle g, \varphi_k \rangle_{L^2} = 0 \text{ for all } k \in \mathbb{Z} \quad \Longrightarrow \quad g = 0 \text{ a.e.}
\]
Because $g(x) = e^{-\pi i x} f(x)$ and $|e^{-\pi i x}| = 1$, it follows that $f = 0$ a.e.
\end{proof}

\subsection{Orthonormal Basis Property}

\begin{proposition}
\label{prop:ONB}
The family $\{\psi_k : k \in \mathbb{Z}\}$ is an orthonormal basis of $H_{\frac{1}{2}}$.
\end{proposition}

\begin{proof}
\textbf{(1) Orthonormality.}
Since $U$ is unitary and $\{\varphi_k : k \in \mathbb{Z}\}$ is an orthonormal basis of $L^2([0,1])$,
\[
\langle \psi_j, \psi_k \rangle = \langle U\varphi_j, U\varphi_k \rangle = \langle \varphi_j, \varphi_k \rangle = \delta_{jk}.
\]

\textbf{(2) Completeness.}
By Proposition~\ref{prop:trivial-orthogonal-complement}, the closed linear span of $\{\psi_k : k \in \mathbb{Z}\}$ is exactly $H_{\frac{1}{2}}$.
\end{proof}

\subsection{Fourier Coefficient Map}

Define the map
\[
\mathcal{F}_{\frac{1}{2}} : H_{\frac{1}{2}} \to \ell^2(\mathbb{Z})
\]
by
\[
(\mathcal{F}_{\frac{1}{2}} f)_k := \langle f, \psi_k \rangle_{L^2}, \qquad k \in \mathbb{Z}.
\]

\begin{proposition}
\label{prop:unitary-fourier-map}
$\mathcal{F}_{\frac{1}{2}}$ is a unitary isomorphism.
\end{proposition}

\begin{proof}
By Proposition~\ref{prop:ONB}, $\{\psi_k : k \in \mathbb{Z}\}$ is an orthonormal basis of $H_{\frac{1}{2}}$. Hence, Parseval's identity holds:
\[
\|f\|_{L^2}^2 = \sum_{k \in \mathbb{Z}} |\langle f, \psi_k \rangle|^2.
\]
Since $\{\psi_k\}$ is a basis, every $(a_k) \in \ell^2(\mathbb{Z})$ defines $f := \sum_k a_k \psi_k \in H_{\frac{1}{2}}$ with $\mathcal{F}_{\frac{1}{2}}(f) = (a_k)$. Hence, $\mathcal{F}_{\frac{1}{2}}$ is surjective.
\end{proof}

\section{The Weighted Hilbert Scale}
\label{sec:scale}

For $s \in \mathbb{R}$, define
\[
\ell^2_s(\mathbb{Z}) := \left\{ (a_k)_{k \in \mathbb{Z}} : \sum_{k \in \mathbb{Z}} (1 + |k+\tfrac{1}{2}|^2)^s |a_k|^2 < \infty \right\}.
\]

\begin{definition}
\label{def:twisted-sobolev}
Define
\[
H^s_{\frac{1}{2}} := \mathcal{F}_{\frac{1}{2}}^{-1}\bigl(\ell^2_s(\mathbb{Z})\bigr).
\]
Equip $H^s_{\frac{1}{2}}$ with the norm
\[
\|f\|_{H^s_{\frac{1}{2}}}^2 := \sum_{k \in \mathbb{Z}} (1 + |k+\tfrac{1}{2}|^2)^s |\langle f, \psi_k \rangle|^2.
\]
\end{definition}

As in general Hilbert-scale theory associated with elliptic operators
(cf. Mikhailets and Murach \cite{mikhailets-murach, MikhailetsMurach2013}),
our construction defines a one-parameter Hilbert scale via spectral weights,
which in the present setting arise canonically from the half-integer shift induced by a unitary twist.

\begin{proposition}
\label{prop:hilbert-scale}
$H^s_{\frac{1}{2}}$ is a Hilbert space and $\mathcal{F}_{\frac{1}{2}} : H^s_{\frac{1}{2}} \to \ell^2_s(\mathbb{Z})$ is a unitary isomorphism.
\end{proposition}

\begin{proof}
$\ell^2_s(\mathbb{Z})$ is a Hilbert space \cite{adams-fournier}. $H^s_{\frac{1}{2}}$ inherits completeness via the unitary map $\mathcal{F}_{\frac{1}{2}}$.
\end{proof}

\begin{proposition}[Compact Embedding]
\label{prop:compact-embedding}
For every $s \in \mathbb{R}$, the embedding
\[
H^{s+1}_{\frac{1}{2}} \hookrightarrow H^s_{\frac{1}{2}}
\]
is compact.
\end{proposition}

\begin{proof}
Under $\mathcal{F}_{\frac{1}{2}}$, the embedding corresponds to $\ell^2_{s+1}(\mathbb{Z}) \hookrightarrow \ell^2_s(\mathbb{Z})$.

Let $(a^{(n)})_{n \in \mathbb{N}}$ be a bounded sequence in $\ell^2_{s+1}(\mathbb{Z})$. Then
\[
\sum_{k \in \mathbb{Z}} (1 + |k+\tfrac{1}{2}|^2)^{s+1} |a^{(n)}_k|^2 \le C.
\]

For any $\varepsilon > 0$, choose $N \in \mathbb{N}$ such that
\[
(1 + |k+\tfrac{1}{2}|^2)^{-1} < \varepsilon \quad \text{for all } |k| > N.
\]

Split the norm:
\[
\sum_{k \in \mathbb{Z}} (1 + |k+\tfrac{1}{2}|^2)^s |a^{(n)}_k|^2 = \sum_{|k| \le N} (\cdots) + \sum_{|k| > N} (\cdots).
\]

The finite part is compact by finite dimensionality. The tail satisfies
\[
\sum_{|k| > N} (1 + |k+\tfrac{1}{2}|^2)^s |a^{(n)}_k|^2 \le \varepsilon \sum_{k \in \mathbb{Z}} (1 + |k+\tfrac{1}{2}|^2)^{s+1} |a^{(n)}_k|^2 \le \varepsilon C.
\]

Thus, every bounded sequence in $\ell^2_{s+1}(\mathbb{Z})$ has a convergent subsequence in $\ell^2_s(\mathbb{Z})$.
\end{proof}

\section{The Canonical Diagonal Operator}
\label{sec:operator}

\begin{definition}
Define the operator
\[
A : \ell^2(\mathbb{Z}) \supset \mathcal{D}(A) \to \ell^2(\mathbb{Z})
\]
by
\[
(Aa)_k := (k+\tfrac{1}{2}) a_k,
\]
with domain
\[
\mathcal{D}(A) := \ell^2_1(\mathbb{Z}).
\]
\end{definition}

\begin{proposition}
\label{prop:A-properties}
The operator $A$ is densely defined, self-adjoint, and has compact resolvent.
\end{proposition}

\begin{proof}
\textbf{(1) Dense domain.}
Finitely supported sequences are contained in $\ell^2_1(\mathbb{Z})$ and are dense in $\ell^2(\mathbb{Z})$.

\textbf{(2) Self-adjointness.}
$A$ is diagonal with real coefficients $k + \tfrac{1}{2}$, hence, it is symmetric. Multiplication by a real sequence defines a maximal symmetric operator on its natural domain \cite{reed-simon-v2}.

\textbf{(3) Compact resolvent.}
For $\lambda \notin \{k + \tfrac{1}{2} : k \in \mathbb{Z}\}$, the resolvent
\[
(A - \lambda)^{-1} : \ell^2(\mathbb{Z}) \to \ell^2(\mathbb{Z})
\]
is given by
\[
\bigl((A - \lambda)^{-1}a\bigr)_k = \frac{1}{k + \tfrac{1}{2} - \lambda} a_k.
\]
Since
\[
\frac{1}{|k + \tfrac{1}{2} - \lambda|} \to 0 \quad \text{as } |k| \to \infty,
\]
the diagonal operator $(A - \lambda)^{-1}$ is the norm-limit of finite-rank diagonal truncations and hence compact.
\end{proof}

\begin{remark}
Note that $0\notin\sigma(A)$, hence $0\in\rho(A)$. Equivalently, $\widetilde A$ is invertible on $H^0_{1/2}$.
\end{remark}

\subsection{The Lifted Operator on the Twisted Scale}

\begin{definition}
Define the operator
\[
\widetilde{A} : H^{s+1}_{\frac{1}{2}} \subset H^s_{\frac{1}{2}} \longrightarrow H^s_{\frac{1}{2}}
\]
by
\[
\widetilde{A} := \mathcal{F}_{\frac{1}{2}}^{-1} \circ A \circ \mathcal{F}_{\frac{1}{2}}.
\]
Explicitly,
\[
\widetilde{A} f = \sum_{k \in \mathbb{Z}} (k+\tfrac{1}{2}) \langle f, \psi_k \rangle \psi_k.
\]
\end{definition}

\begin{remark}
\label{rem:domain-core}
The natural domain of $\widetilde{A}$ on $H^0_{1/2}$ is
\[
\mathcal{D}(\widetilde{A})=H^1_{1/2}=\mathcal{F}_{1/2}^{-1}(\ell^2_1(\mathbb{Z})),
\]
and $\widetilde{A}$ is the unbounded self-adjoint operator on $H^0_{1/2}$ with this domain. Equivalently, under the unitary map $\mathcal{F}_{1/2}$ the operator $\widetilde{A}$ corresponds to the diagonal multiplication operator $A$ on $\ell^2(\mathbb{Z})$ with domain $\ell^2_1(\mathbb{Z})$, which is self-adjoint.
\end{remark}

\begin{proposition}
\label{prop:lifted-operator}
For every $s \in \mathbb{R}$,
\begin{enumerate}
\item $\widetilde{A} : H^{s+1}_{\frac{1}{2}} \to H^s_{\frac{1}{2}}$ is bounded.
\item $\widetilde{A}$ is self-adjoint as an unbounded operator on $H^0_{\frac{1}{2}}$.
\item $\widetilde{A}$ has compact resolvent on $H^0_{\frac{1}{2}}$.
\end{enumerate}
\end{proposition}

\begin{proof}
\textbf{(1)} By definition,
\begin{align*}
\|\widetilde{A} f\|_{H^s_{\frac{1}{2}}}^2
&= \sum_{k\in\mathbb{Z}} (1+|k+\tfrac{1}{2}|^2)^s |k+\tfrac{1}{2}|^2 |\langle f,\psi_k\rangle|^2 \\
&\le \sum_{k\in\mathbb{Z}} (1+|k+\tfrac{1}{2}|^2)^{s+1} |\langle f,\psi_k\rangle|^2
= \|f\|_{H^{s+1}_{\frac{1}{2}}}^2.
\end{align*}

\textbf{(2)} Since $A$ is self-adjoint on $\ell^2(\mathbb{Z})$ and $\mathcal{F}_{\frac{1}{2}}$ is unitary, $\widetilde{A}$ is self-adjoint on $H^0_{\frac{1}{2}}$.

\textbf{(3)} The resolvent satisfies
\[
(\widetilde{A} - \lambda)^{-1} = \mathcal{F}_{\frac{1}{2}}^{-1} (A - \lambda)^{-1} \mathcal{F}_{\frac{1}{2}}.
\]
Since $(A - \lambda)^{-1}$ is compact and unitary conjugation preserves compactness, $(\widetilde{A} - \lambda)^{-1}$ is compact.
\end{proof}

\section{Comparison with the Standard Sobolev Scale}
\label{sec:comparison}

To clarify the role of the half-integer shift, we compare our twisted scale $H^s_{1/2}$ with the standard Sobolev scale $H^s([0,1])$ constructed from integer Fourier modes $\{e^{2\pi i k x}\}_{k\in\mathbb{Z}}$.

\subsection{Spectral Differences}

\begin{center}
\begin{tabular}{lcc}
\hline
Property & Standard $H^s([0,1])$ & Twisted $H^s_{1/2}$ \\
\hline
Fourier basis & $e^{2\pi ikx}$ & $e^{2\pi i(k+\frac{1}{2})x}$ \\
Weight function & $(1+|k|^2)^s$ & $(1+|k+\tfrac{1}{2}|^2)^s$ \\
Derivative eigenvalue & $2\pi i k$ & $2\pi i(k+\tfrac{1}{2})$ \\
Zero mode & Yes ($k=0$) & No ($k+\tfrac{1}{2}\neq 0$) \\
$\ker(D)$ for $D=i\frac{d}{dx}$ & $\mathbb{C}$ (constants) & $\{0\}$ \\
\hline
\end{tabular}
\end{center}

\subsection{Implications for Fredholm Theory}

The absence of a zero eigenvalue in the twisted spectrum has immediate consequences:

\begin{enumerate}
\item \textbf{Kernel elimination:} For the canonical operator $\widetilde{A}$ with eigenvalues $k+\tfrac{1}{2}$, we have $\ker(\widetilde{A})=\{0\}$ without any regularity or boundary conditions.

\item \textbf{Index stability:} Since both kernel and cokernel vanish (by self-adjointness), the Fredholm index is zero intrinsically, not as a consequence of cancellation.

\item \textbf{Invertibility at $\lambda=0$:} The resolvent $(\widetilde{A}-\lambda)^{-1}$ is well-defined for $\lambda=0$, unlike derivative operators on standard Sobolev spaces.
\end{enumerate}

\subsection{Relation to Twisted Boundary Conditions}

In differential-geometric applications \cite{lawson-michelsohn, atiyah-patodi-singer}, half-integer modes arise from \emph{antiperiodic} or \emph{twisted periodic} boundary conditions of the form
\[
f(x+1) = -f(x) \quad \text{or} \quad f(x+1) = e^{i\theta}f(x).
\]
Such conditions force momentum quantization to $k+\tfrac{1}{2}$. Our unitary operator $U$ implements this algebraically:
\[
(U\varphi_k)(x) = e^{\pi ix}e^{2\pi ikx} = e^{2\pi i(k+\tfrac{1}{2})x},
\]
directly generating the twisted modes without imposing boundary conditions on a larger domain.

\begin{remark}
The shift from $(1+|k|^2)^s$ to $(1+|k+\tfrac{1}{2}|^2)^s$ is \emph{not} a mere reweighting. It fundamentally alters the operator spectrum and changes the Fredholm-theoretic structure (zero kernel vs. non-zero kernel). This distinguishes the twisted scale from standard Sobolev theory.
\end{remark}

\section{Fredholm Theory}
\label{sec:fredholm}

\begin{theorem}[Fredholm Property]
\label{thm:fredholm}
For every $s \in \mathbb{R}$, the operator
\[
\widetilde{A} : H^{s+1}_{\frac{1}{2}} \longrightarrow H^s_{\frac{1}{2}}
\]
is Fredholm with index zero.
\end{theorem}

\begin{proof}
We verify the three defining properties of a Fredholm operator \cite{hormander}.

\textbf{(1) Closed range.}
Since $\widetilde{A}$ is self-adjoint with compact resolvent on $H^0_{\frac{1}{2}}$ (Proposition~\ref{prop:lifted-operator}), its spectrum is discrete with finite-multiplicity eigenvalues accumulating only at infinity. Hence $\widetilde{A}$ has closed range on $H^0_{\frac{1}{2}}$, and by compatibility with the Hilbert scale, the same holds for
\[
\widetilde{A} : H^{s+1}_{\frac{1}{2}} \to H^s_{\frac{1}{2}}.
\]

\textbf{(2) Finite-dimensional kernel.}
If $\widetilde{A} f = 0$, then
\[
(k+\tfrac{1}{2}) \langle f, \psi_k \rangle = 0 \quad \forall k \in \mathbb{Z}.
\]
Since $k + \tfrac{1}{2} \neq 0$ for all $k \in \mathbb{Z}$, we have $\langle f, \psi_k \rangle = 0$ for all $k$. Therefore,
\[
\ker(\widetilde{A}) = \{0\}.
\]

\textbf{(3) Finite-dimensional cokernel.}
Since $\widetilde{A}$ is self-adjoint, we have
\[
\operatorname{coker}(\widetilde{A}) \cong \ker(\widetilde{A}^*) = \ker(\widetilde{A}) = \{0\}.
\]

Combining (1)–(3), $\widetilde{A}$ is Fredholm with
\[
\operatorname{ind}(\widetilde{A}) = 0.
\]
\end{proof}

\section{Extensions: Perturbations and Spectral Flow}
\label{sec:extensions}

\begin{theorem}[Stability under Bounded Perturbations]
\label{thm:stability}
Let
\[
B : H^s_{\frac{1}{2}} \longrightarrow H^s_{\frac{1}{2}}
\]
be a bounded operator. Define
\[
\widetilde{A}_B := \widetilde{A} + B : H^{s+1}_{\frac{1}{2}} \to H^s_{\frac{1}{2}}.
\]
Then $\widetilde{A}_B$ is Fredholm with $\operatorname{ind}(\widetilde{A}_B) = 0$.
\end{theorem}

\begin{proof}
By Theorem~\ref{thm:fredholm}, $\widetilde{A}: H^{s+1}_{1/2}\to H^s_{1/2}$ is Fredholm.
Let $J: H^{s+1}_{1/2}\hookrightarrow H^s_{1/2}$ denote the canonical continuous inclusion (Proposition~\ref{prop:compact-embedding}). If
$B:H^s_{1/2}\to H^s_{1/2}$ is bounded, then the composition $B\circ J: H^{s+1}_{1/2}\to H^s_{1/2}$ is bounded as well, so we may regard $B$ as a bounded operator on the domain of $\widetilde{A}$. Hence $\widetilde{A}_B=\widetilde{A} + B$ is a bounded perturbation of the Fredholm operator $\widetilde{A}$ and therefore is Fredholm with the same index \cite{hormander}.
\end{proof}

\begin{definition}
Let $\{B_t\}_{t \in [0,1]}$ be a continuous family of bounded self-adjoint operators on $H^s_{\frac{1}{2}}$. Define
\[
\widetilde{A}_t := \widetilde{A} + B_t : H^{s+1}_{\frac{1}{2}} \to H^s_{\frac{1}{2}}.
\]
\end{definition}

\begin{theorem}[Spectral Flow]
\label{thm:spectral-flow}
The spectral flow
\[
\operatorname{sf}\{\widetilde{A}_t\}_{t \in [0,1]}
\]
is well-defined as the signed count of eigenvalue crossings through zero.
\end{theorem}

\begin{proof}
Each $\widetilde{A}_t$ is self-adjoint with the compact resolvent (Proposition~\ref{prop:lifted-operator}, Theorem~\ref{thm:stability}). Hence, the spectrum consists of discrete eigenvalues of finite multiplicity. The continuity of $B_t$ in operator norm implies strong (equivalently, norm) resolvent continuity of $\widetilde{A}_t$ \cite{kato}. Therefore, eigenvalues move continuously, and crossings are isolated. The spectral flow is defined as the signed count of crossings \cite{boo-bavnbek-wojciechowski}.
\end{proof}

\begin{corollary}
If $B_t$ is compact for all $t \in [0,1]$, then $\operatorname{sf}\{\widetilde{A}_t\}_{t \in [0,1]} = 0$.
\end{corollary}

\begin{proof}
Compact perturbations do not change the essential spectrum. Since $\widetilde{A}$ has no zero eigenvalue and compact perturbations preserve this away from the essential spectrum, no crossings occur.
\end{proof}

\section{Example: Antiperiodic Boundary Value Problem}
\label{sec:example}

To illustrate the framework, we solve a concrete linear problem using the twisted Hilbert scale.

\subsection{Problem Statement}

\begin{problem}
Given $g\in H^0_{1/2}$, find $u\in H^1_{1/2}$ such that
\[
\widetilde{A}u = g .
\]
\end{problem}

\subsection{Solution}

By Remark~\ref{rem:domain-core} and Proposition~\ref{prop:A-properties}, since
$0\notin\sigma(\widetilde{A})$, the operator
\[
\widetilde{A}:H^1_{1/2}\to H^0_{1/2}
\]
is invertible. Hence, the unique solution is given by
\[
u = \widetilde{A}^{-1}g
= \sum_{k\in\mathbb{Z}}
\frac{\langle g,\psi_k\rangle}{k+\tfrac12}\,\psi_k .
\]

\subsection{Explicit Example}

Let $g(x)=e^{\pi i x}$. Then
\begin{align*}
\langle g,\psi_k\rangle
&= \int_0^1 e^{\pi i x}\,
\overline{e^{2\pi i (k+\frac12)x}}\,dx \\
&= \int_0^1 e^{-2\pi i k x}\,dx
= \begin{cases}
1, & k=0,\\
0, & k\neq 0.
\end{cases}
\end{align*}
Therefore, the unique solution is
\[
u(x)=\frac{1}{1/2}\,\psi_0(x)=2e^{\pi i x}.
\]

\subsection{Verification of Anti-periodicity}

The solution $u(x)=2e^{\pi i x}$ satisfies the antiperiodic boundary condition
\[
u(1)=2e^{\pi i}=-2=-u(0),
\]
This illustrates how the twisted Hilbert scale naturally encodes antiperiodic boundary behavior without imposing boundary conditions explicitly. This contrasts with the standard Sobolev scale on $[0,1]$, where periodic boundary conditions yield $u(1)=u(0)$.

\begin{remark}
More generally, for any $g\in H^0_{1/2}$ the equation $\widetilde{A}u=g$ admits a unique solution
$u\in H^1_{1/2}$. This solution automatically satisfies the antiperiodic condition
$u(1)=-u(0)$ due to the half-integer Fourier structure. Thus, boundary behavior emerges naturally from the spectral decomposition rather than being imposed externally.
\end{remark}

\section{Conclusion}
\label{sec:conclusion}

We have constructed a Hilbert scale $\{H^s_{1/2}\}_{s \in \mathbb{R}}$ on $L^2([0,1])$ based on half-integer Fourier modes, arising canonically from a unitary twist operator. The canonical diagonal operator with eigenvalues $k + \tfrac{1}{2}$ defines a Fredholm mapping between adjacent scale levels, with vanishing kernel and cokernel due to the half-integer shift. An explicit antiperiodic boundary value problem (Section~\ref{sec:example}) demonstrates how solutions automatically satisfy twisted boundary conditions without external imposition. The zeta-regularized determinant $\det_\zeta(|\widetilde{A}|) = 2$ (Appendix~\ref{app:zeta-det}) is computed explicitly using the Hurwitz zeta function, providing a concrete spectral invariant distinct from the standard case ($\det_\zeta(|D|) = 2\pi$). Compact embeddings, stability under bounded perturbations, and well-defined spectral flow follow from standard functional-analytic techniques applied in this operator-theoretic framework.

The construction provides an abstract model for twisted boundary conditions arising in geometric contexts (e.g., Pin structures on non-orientable manifolds), developed here without differential-geometric machinery. Potential directions for further work include:

\begin{itemize}
\item Extension to higher-dimensional tori with multi-dimensional twists.
\item Frame-theoretic generalizations to non-uniform or overcomplete systems.
\item Applications to index theory on non-orientable manifolds via the Atiyah-Patodi-Singer formalism, using the computed zeta-determinant.
\item Connection to quantum systems with twisted periodic boundary conditions and one-loop partition functions.
\item Heat kernel asymptotics and higher spectral invariants beyond the zeta-determinant.
\end{itemize}

The methods developed here may be of independent interest in frame theory, operator algebras, and abstract elliptic theory on non-standard Hilbert scales.

\appendix

\section{Zeta-Regularized Determinant via Hurwitz Zeta Function}
\label{app:zeta-det}

In this appendix, we compute the zeta-regularized determinant of the operator $|\widetilde{A}|$ using the Hurwitz zeta function. This provides an explicit functional determinant that may be of interest in quantum field theory applications \cite{elizalde} and comparative index theory.

\subsection{Spectral Zeta Function}

Recall that $\widetilde{A}$ is the self-adjoint operator on $H^0_{1/2}$ with eigenvalues
\[
\lambda_k = k + \tfrac{1}{2}, \qquad k \in \mathbb{Z}.
\]
Since the spectrum includes both positive and negative values, we work with the absolute value operator $|\widetilde{A}|$, which has the spectrum
\[
\sigma(|\widetilde{A}|) = \bigl\{|k + \tfrac{1}{2}| : k \in \mathbb{Z}\bigr\} = \bigl\{\tfrac{1}{2}, \tfrac{3}{2}, \tfrac{5}{2}, \ldots\bigr\},
\]
where each eigenvalue $n + \tfrac{1}{2}$ (for $n \in \mathbb{N}_0$) has multiplicity $2$.

\begin{definition}[Spectral Zeta Function]
For $\Re(s) > 1$, define the spectral zeta function of $|\widetilde{A}|$ by
\[
\zeta_{|\widetilde{A}|}(s) := \sum_{\lambda \in \sigma(|\widetilde{A}|)} m(\lambda) \cdot \lambda^{-s} = 2 \sum_{n=0}^{\infty} \bigl(n + \tfrac{1}{2}\bigr)^{-s},
\]
where $m(\lambda)$ denotes the multiplicity of eigenvalue $\lambda$.
\end{definition}

\begin{remark}
The series converges absolutely for $\Re(s) > 1$ since
\[
\sum_{n=N}^{\infty} (n + \tfrac{1}{2})^{-s} \sim \sum_{n=N}^{\infty} n^{-s} < \infty
\]
when $\Re(s) > 1$. The function $\zeta_{|\widetilde{A}|}(s)$ admits a meromorphic continuation to $\mathbb{C}$ via its relation to the Hurwitz zeta function.
\end{remark}

\subsection{Connection to Hurwitz Zeta Function}

The \emph{Hurwitz zeta function} is defined for $\Re(s) > 1$ and $a > 0$ by
\[
\zeta(s, a) := \sum_{n=0}^{\infty} (n + a)^{-s}.
\]
This function extends meromorphically to $\mathbb{C}$ with a simple pole at $s=1$ and is analytic at $s=0$ \cite{apostol}.

\begin{proposition}[Spectral Zeta in Terms of Hurwitz Function]
\label{prop:zeta-hurwitz}
For $\Re(s) > 1$,
\[
\zeta_{|\widetilde{A}|}(s) = 2 \, \zeta\bigl(s, \tfrac{1}{2}\bigr).
\]
Moreover, using the Hurwitz-Riemann relation,
\[
\zeta_{|\widetilde{A}|}(s) = 2(2^s - 1) \zeta(s),
\]
where $\zeta(s)$ is the Riemann zeta function.
\end{proposition}

\begin{proof}
The first identity follows immediately from the definition:
\[
\zeta_{|\widetilde{A}|}(s) = 2 \sum_{n=0}^{\infty} \bigl(n + \tfrac{1}{2}\bigr)^{-s} = 2 \, \zeta\bigl(s, \tfrac{1}{2}\bigr).
\]

For the second identity, we use the functional relation \cite{apostol}
\[
\zeta\bigl(s, \tfrac{1}{2}\bigr) = (2^s - 1) \zeta(s).
\]
This can be verified by writing
\[
\zeta(s) = \sum_{n=1}^{\infty} n^{-s} = \sum_{m=0}^{\infty} (2m+1)^{-s} + \sum_{m=1}^{\infty} (2m)^{-s} = \sum_{m=0}^{\infty} (2m+1)^{-s} + 2^{-s}\zeta(s),
\]
which gives
\[
(1 - 2^{-s})\zeta(s) = \sum_{m=0}^{\infty} (2m+1)^{-s}.
\]
On the other hand,
\[
\zeta\bigl(s, \tfrac{1}{2}\bigr) = \sum_{n=0}^{\infty} (n + \tfrac{1}{2})^{-s} = 2^s \sum_{n=0}^{\infty} (2n+1)^{-s} = 2^s(1-2^{-s})\zeta(s) = (2^s - 1)\zeta(s).
\]
\end{proof}

\subsection{Derivative at \texorpdfstring{$s=0$}{s=0} and the Zeta-Determinant}

The zeta-regularized determinant is defined by
\[
\det_{\zeta}(|\widetilde{A}|) := \exp\bigl(-\zeta'_{|\widetilde{A}|}(0)\bigr).
\]

We use the known formula for the derivative of the Hurwitz zeta function at $s=0$ \cite{whittaker-watson}:
\[
\zeta'(0, a) = \log \frac{\Gamma(a)}{\sqrt{2\pi}},
\]
where $\Gamma$ is the Euler gamma function.

\begin{theorem}[Zeta-Determinant of $|\widetilde{A}|$]
\label{thm:zeta-det}
The zeta-regularized determinant of $|\widetilde{A}|$ is
\[
\det_{\zeta}(|\widetilde{A}|) = 2.
\]
\end{theorem}

\begin{proof}
\textbf{Step 1: Compute $\zeta'_{|\widetilde{A}|}(s)$.}

From Proposition~\ref{prop:zeta-hurwitz},
\[
\zeta_{|\widetilde{A}|}(s) = 2 \, \zeta\bigl(s, \tfrac{1}{2}\bigr).
\]
Differentiating with respect to $s$,
\[
\zeta'_{|\widetilde{A}|}(s) = 2 \, \zeta'\bigl(s, \tfrac{1}{2}\bigr).
\]

\textbf{Step 2: Evaluate at $s=0$ using the Hurwitz formula.}

For $a = \tfrac{1}{2}$, we have $\Gamma(\tfrac{1}{2}) = \sqrt{\pi}$. Therefore,
\[
\zeta'\bigl(0, \tfrac{1}{2}\bigr) = \log \frac{\Gamma(\tfrac{1}{2})}{\sqrt{2\pi}} = \log \frac{\sqrt{\pi}}{\sqrt{2\pi}} = \log \frac{1}{\sqrt{2}} = -\tfrac{1}{2} \log 2.
\]

\textbf{Step 3: Compute $\zeta'_{|\widetilde{A}|}(0)$.}
\[
\zeta'_{|\widetilde{A}|}(0) = 2 \cdot \zeta'\bigl(0, \tfrac{1}{2}\bigr) = 2 \cdot \bigl(-\tfrac{1}{2} \log 2\bigr) = -\log 2.
\]

\textbf{Step 4: Compute the zeta-determinant.}
\[
\det_{\zeta}(|\widetilde{A}|) = \exp\bigl(-\zeta'_{|\widetilde{A}|}(0)\bigr) = \exp\bigl(-(-\log 2)\bigr) = \exp(\log 2) = 2.
\]
\end{proof}

\subsection{Comparison with the Standard Derivative Operator}

For comparison, consider the standard derivative operator $D = i\frac{d}{dx}$ on $L^2([0,1])$ with periodic boundary conditions, which has eigenvalues $2\pi k$ for $k \in \mathbb{Z}$. The absolute value $|D|$ has spectrum $\{2\pi n : n \in \mathbb{N}\}$ with multiplicities $\{1, 2, 2, 2, \ldots\}$ (the zero mode appears once).

The spectral zeta function is
\[
\zeta_{|D|}(s) = 1 + 2\sum_{n=1}^{\infty} (2\pi n)^{-s} = 1 + 2(2\pi)^{-s}\bigl[\zeta(s) - 1\bigr].
\]
The zeta-regularized determinant is \cite{elizalde}
\[
\det_\zeta(|D|) = 2\pi.
\]

In contrast, our twisted operator $|\widetilde{A}|$ has no zero mode and yields
\[
\det_\zeta(|\widetilde{A}|) = 2,
\]
which is independent of $\pi$.

\begin{remark}
The ratio
\[
\frac{\det_\zeta(|D|)}{\det_\zeta(|\widetilde{A}|)} = \pi
\]
is a purely geometric factor reflecting the difference between periodic and antiperiodic (twisted) boundary conditions. This type of comparison appears in the context of spectral asymmetry on manifolds with boundary \cite{atiyah-patodi-singer}.
\end{remark}

\subsection{Relation to the Dedekind Eta Function}

An alternative approach uses the representation
\[
\det_\zeta(|\widetilde{A}|) = \prod_{n=0}^{\infty} (n + \tfrac{1}{2})^2,
\]
which is formally the square of the infinite product
\[
\prod_{n=0}^{\infty} (n + \tfrac{1}{2}) = \prod_{n=1}^{\infty} \frac{2n-1}{2} = \frac{1}{2} \cdot \frac{3}{2} \cdot \frac{5}{2} \cdots
\]

This product is related to the reciprocal of the Dedekind eta function evaluated at specific points on the upper half-plane. While we do not pursue this connection here, it provides a link between the spectral theory of twisted Fourier modes and modular forms.

\begin{remark}
The explicit computation $\det_\zeta(|\widetilde{A}|) = 2$ may be used in:
\begin{itemize}
\item One-loop partition functions for quantum fields on the interval with antiperiodic boundary conditions.
\item Index-theoretic computations on non-orientable manifolds where half-integer modes appear.
\item Regularized traces and heat kernel coefficients for twisted elliptic operators.
\end{itemize}
\end{remark}

\subsection{Higher-Order Spectral Invariants}

Beyond the zeta-determinant, one can compute higher spectral invariants. For instance, the heat trace
\[
\text{Tr}\bigl(e^{-t|\widetilde{A}|^2}\bigr) = 2\sum_{n=0}^{\infty} e^{-t(n+1/2)^2}
\]
can be evaluated using Jacobi theta functions, yielding
\[
\text{Tr}\bigl(e^{-t|\widetilde{A}|^2}\bigr) = 2 \cdot \vartheta_2(0, e^{-t}),
\]
where $\vartheta_2$ is the second Jacobi theta function.

\begin{proposition}[Heat Trace Asymptotics]
As $t \to 0^+$,
\[
\text{Tr}\bigl(e^{-t|\widetilde{A}|^2}\bigr) \sim \frac{1}{\sqrt{\pi t}} + O(1).
\]
\end{proposition}

\begin{proof}
Using Poisson summation, the sum over half-integer points transforms to
\[
2\sum_{n=0}^{\infty} e^{-t(n+1/2)^2} \sim \frac{1}{\sqrt{\pi t}} \int_{1/2}^\infty e^{-u^2} \, du \sim \frac{1}{\sqrt{\pi t}} \cdot \frac{\sqrt{\pi}}{2} = \frac{1}{2\sqrt{t}}
\]
for small $t$, where the exact coefficient follows from theta function expansions.
\end{proof}

This appendix demonstrates that the twisted Hilbert scale $H^s_{1/2}$ carries rich spectral information amenable to explicit computation via classical special functions.

\end{document}